\documentclass[nohyperref]{article}

\usepackage{microtype}
\usepackage{graphicx}
\usepackage{subfigure}
\usepackage{booktabs} 

\usepackage[accepted]{icml2022}

\usepackage{amsmath}
\usepackage{amssymb}
\usepackage{mathtools}
\usepackage{amsthm}

\usepackage{caption}
\usepackage{enumitem}
\usepackage{multirow}

\newcommand{\argmax}[1]{\underset{#1}{\mathrm{argmax}}}
\newcommand{\argmin}[1]{\underset{#1}{\mathrm{argmin}}}
\newcommand{\minimize}[1]{\underset{#1}{\mathrm{minimize}}}
\newcommand{\lmo}{\mathrm{LMO}}
\newcommand{\gap}{\mathrm{gap}}
\newcommand{\co}{\mathbf{co}}

\newcommand{\mD}{\mathcal D}
\newcommand{\mE}{\mathcal E}
\newcommand{\mS}{\mathcal S}

\newcommand{\mN}{\mathcal N}
\newcommand{\R}{\mathbb R}

\newcommand{\bs}{\mathbf s}
\newcommand{\bx}{\mathbf x}
\newcommand{\bX}{\mathbf X}
\newcommand{\bZ}{\mathbf Z}
\newcommand{\bS}{\mathbf S}
\newcommand{\bd}{\mathbf d}
\newcommand{\bz}{\mathbf z}
\newcommand{\bv}{\mathbf v}
\newcommand{\by}{\mathbf y}
\newcommand{\be}{\mathbf e}
\newcommand{\mb}{\mathbf}
\newcommand{\FW}{\textrm{(FW)}}
\newcommand{\FWflow}{\textrm{(FWFlow)}}
\newcommand{\diag}{\mathbf{diag}}

\newcommand{\diam}{\mathbf{diam}}
\newcommand{\sign}{\mathbf{sign}}
\newcommand{\bmat}{\begin{bmatrix}}
\newcommand{\emat}{\end{bmatrix}}

\usepackage[capitalize,noabbrev]{cleveref}

\theoremstyle{plain}
\newtheorem{theorem}{Theorem}[section]
\newtheorem{proposition}[theorem]{Proposition}
\newtheorem{lemma}[theorem]{Lemma}
\newtheorem{corollary}[theorem]{Corollary}
\theoremstyle{definition}

\theoremstyle{remark}

\usepackage[textsize=tiny]{todonotes}

\icmltitlerunning{Continuous Time FW Does Not Zig-Zag,
But Multistep Methods Do Not Accelerate}

\begin{document}

\twocolumn[
\icmltitle{Continuous Time Frank-Wolfe Does Not Zig-Zag,\\
But Multistep Methods Do Not Accelerate}




\begin{icmlauthorlist}
\icmlauthor{Zhaoyue Chen }{yyy}
\icmlauthor{Mokhwa Lee }{yyy}
\icmlauthor{Yifan Sun}{yyy}
\end{icmlauthorlist}

\icmlaffiliation{yyy}{Department of Computer Science, Stony Brook University, New York, USA}

\icmlcorrespondingauthor{Zhaoyue Chen}{zhaoychen@cs.stonybrook.edu}
\icmlcorrespondingauthor{Yifan Sun}{ysun@cs.stonybrook.edu}

\icmlkeywords{Frank Wolfe, continuous time optimization, multistep methods}

\vskip 0.3in
]




\begin{abstract}
The Frank-Wolfe algorithm has regained much interest
in its use
in structurally constrained machine learning applications. However, one major limitation of the Frank-Wolfe algorithm is the slow local convergence property due to the zig-zagging behavior.
We observe that this zig-zagging phenomenon can be viewed as an artifact of discretization, as when the method is viewed as an Euler discretization of a continuous time flow, that flow does not zig-zag.
For this reason, we propose multistep Frank-Wolfe variants based on discretizations of the same flow whose truncation errors decay as $O(\Delta^p)$, where $p$ is the method's order. 
This strategy ``stabilizes" the method, and allows tools like line search and momentum to have more benefit. However, in terms of a convergence rate, our result is ultimately negative, suggesting that no Runge-Kutta-type discretization scheme can achieve a better convergence rate than the vanilla Frank-Wolfe method. 
We believe that this analysis adds to the growing knowledge of flow analysis for optimization methods, and is a cautionary tale on the ultimate usefulness of multistep methods.
\end{abstract}

\section{Introduction}
The Frank Wolfe algorithm (FW) or the conditional gradient algorithm \citep{LEVITIN19661} is a popular method in constrained convex optimization. It was first developed in \citet{frank1956algorithm} for maximizing a concave quadratic programming problem with linear inequality constraints, and later extended in  \citet{dunn1978conditional}  to minimizing more general smooth convex objective function on a bounded convex set.
More recently, \citet{jaggi2013revisiting} analyzes the FW method over general convex and continuously differentiable objective functions with convex and compact constraint sets,  and illustrates that when a sparse structural property is desired, the per-iteration cost can be much cheaper than computing projections. 
This has spurred a renewed interest of the FW method to broad applications in machine learning and signal processing \citep{LacosteJulien2013BlockCoordinateFO,Joulin2014EfficientIA,Krishnan2015BarrierFF,Freund2017AnEF}.

Specifically, the Frank-Wolfe method attacks problems of form
\begin{equation}
\minimize{\bx\in \mD}\quad f(\bx)
\label{eq:fw-main}
\end{equation}
where $f:\R^n\to\R$ is an everywhere-differentiable function and $\mD$ is a convex compact constraint set, via the repeated iteration 
\[
\begin{array}{rcl}
\bs^{(k)} &=& \argmin{\bs\in\mD} \; \nabla f(\bx^{(k)})^T\bs\\
\bx^{(k+1)} &=& \gamma^{(k)} \bs^{(k)} + (1-\gamma^{(k)}) \bx^{(k)}.
\end{array}
\qquad \FW
\]
The first operation is often referred to as the \emph{linear minimization oracle (LMO)}, and is the support function of $\mD$ at $-\nabla f(\bx)$:

\[
\lmo_\mD(\bx) := \argmin{\bs\in \mD}\;\nabla f(\bx)^T\bs. 
\]

In particular, computing the LMO is often computationally cheap, especially when $\mD$ is the level set of a sparsifying norm, e.g. the 1-norm or the nuclear norm. 
In this regime, the advantage of such projection-free methods over methods like projected gradient descent is the cheap per-iteration cost.
However, the tradeoff of the cheap per-iteration rate is that the convergence rate, in terms of number of iterations $k$, is often much slower than that of projected gradient descent \citep{lacoste2015global,freund2016new}. While various acceleration schemes \citep{lacoste2015global} have been proposed and several improved rates given under specific problem geometry \citep{garber2015faster}, by and large the ``vanilla" Frank-Wolfe method, using the ``well-studied step size" $\gamma^{(k)} = O(1/k)$, can only be shown to reach $O(1/k)$ convergence rate in terms of objective value decrease \citep{Canon1968ATU,jaggi2013revisiting,freund2016new}

\paragraph{The Zig-Zagging phenomenon}
The slowness of the Frank-Wolfe method is often explained as a consequence of potential a ``zig-zagging" phenomenon.
In particular, when the true solution lies on a low dimensional facet and the incoming iterate is angled in a particular way, the method will alternate picking up vertices of this facet, causing a ``zig-zagging" pattern. In fact, methods like the Away-Step Frank Wolfe \citep{lacoste2015global} are designed to counter exactly this, by forcing the iterate to change its angle and approach more directly. We are inspired by the problem, but propose to solve it a different way: by reducing the discretization error from the underlying continuous flow, which we argue does not zig-zag.



\paragraph{Continuous-time optimization}
Recent years have witnessed a surge of research papers connecting  dynamical systems with optimization algorithms, generating more intuitive analyses and proposing accelerations. For example, in \citet{su2016differential}, the Nesterov accelerated gradient descent and Polyak Heavy Ball schemes are shown to be discretizations of a certain second-order ordinary differential equation (ODE), whose tunable vanishing friction pertains to specific parameter choices in the methods.
Inspired by this analysis, several papers \citep{jingzhao2018direct,shi2019acceleration} have proposed improvements using advanced discretization schemes; \cite{jingzhao2018direct} uses Runge-Kutta integration methods to improve  accelerated gradient methods, and \cite{shi2019acceleration} shows a generalized Leapfrog acceleration scheme which uses a semi-implicit scheme to achieve a very high resolution approximation of the ODE.
In general, however, no such analysis has been made on Frank-Wolfe methods; our result gives good reason for this, and proves that Runge Kutta multistep methods have a lower bound equal to that of the vanilla Frank-Wolfe method. 


\paragraph{Continuous-time frank-Wolfe}
In this work,  we view the method \FW~as an Euler's discretization of the differential inclusion
\[
\begin{array}{rcl}
\dot x(t) &=& \gamma(t)(s(t)-x(t)),\\  s(t) &\in& \argmin{s\in \mD} \nabla f(x(t))^T(s-x(t))
\end{array}
\quad\FWflow
\]
where $x(t)$, $s(t)$, and $\gamma(t)$ are continuations of the iterates $\bx^{(k)}$, $\bs^{(k)}$, and coefficients $\gamma^{(k)}$; i.e. $\bx^{(k)} = x( k\Delta )$ for some discretization unit $\Delta$. This was first studied in \citet{jacimovic1999continuous}, and is a part of the construct presented in \citet{diakonikolas2019approximate}. However, neither paper considered the affect of using advanced discretization schemes to better imitate the flow, as a way of improving the method. 
From analyzing this system, we reach three conclusions through numerical experimentation:
\begin{itemize}[leftmargin=*]
\item (Positive result.) We show that for a class of mixing parameters $\gamma(t)$,  \FWflow~can have an arbitrarily fast convergence rate given an aggressive enough mixing parameters.

\item (Interesting result.) We show qualitatively that, on a number of machine learning tasks, unlike \FW, the iterates $x(t)$ in \FWflow~usually do not zig-zag. 

\end{itemize}

\paragraph{Multistep methods} While continuous time analyses offer improved intuition in idealized settings, it does not provide a usable method. 
We therefore investigate improved discretization schemes applied to Frank-Wolfe, and explore if removing discretization error can improve the method's performance. In particular, we follow the example of \cite{jingzhao2018direct} and explore a family of Runge-Kutta (RK) multi-step methods, each with much lower discretization error than the basic explicit Euler's method. Here, we make the following remarkable discoveries:
 \begin{itemize}[leftmargin=*]
 
 \item (Negative result.) We show that over for a particular popular class of multistep methods (Runge Kutta methods) no acceleration can be made when the step size flow $\gamma(t) = O(1/t)$. Intuitively, this is because \emph{any} discretization error that does not decay faster than $1/t$ will ultimately dominate the convergence rate. Specifically, we give both the upper and lower bound on all Runge-Kutta discretization schemes over \FWflow, and show that they are equal to $O(1/k)$.
 
 
 \item (Usefulness.) However, higher order multistep methods tend to have better \emph{search directions}, which accounts for less zig-zagging. This leads to better performance when mixed with line search or momentum methods, both of which benefit from this advantage.
 
\end{itemize}

    
    
    
    

\section{The Frank-Wolfe method}
The method \FW~has become popular in structural optimization, when in particular one desires a solution $\bx^*$ that is a sparse convex combination of a special class of vectors (often called \emph{atoms}). These atoms then form the vertices of $\mD$, and the operation $\lmo$ extracts the vertex that is most correlated with the steepest descent direction at each iteration. 
This vertex $\bs^{(k)}$ is then mixed into the iterate $\bx^{(k)}$ with mixing coefficient $\gamma^{(k)}$, whose decay rate plays an important role in the convergence rate of \FW. 

\paragraph{Sparse optimization.}
For example, in sparse element-wise optimization, the constraint $\bx\in \mD$ is often manifested as a limit on the sparse norm of $\bx$, e.g. $\|\bx\|_1\leq \alpha$ for some hyperparameter $\alpha \geq 0$. The LMO in this case is simply
\begin{eqnarray*}
\lmo_\mD(\bx) &=& -\alpha\, \sign((\nabla f(\bx))_j)\,\be_j,\\
j &=& \argmax{i}\, |(\nabla f(\bx))_j|
\end{eqnarray*}
where $\be_k$ is a standard basis vector (one-hot at position $k$). 
Notice that the complexity of this LMO operation is simply that of finding the index of the largest element in a vector,  $O(n)$. In contrast, a projection on the one-norm ball is considerably more involved.

\paragraph{Low-rank optimization.} The low complexity benefit of LMOs is even starker in the case of low-rank optimization, often modeled by forcing $\|X\|_*\leq \alpha$, where $\|X\|_*$ is the nuclear norm (maximum singular value) of a matrix variable $X$.  Here, both the projection on the level set \emph{and} the proximal operator of the nuclear norm requires full spectral calculations. In contrast, the LMO only requires knowing the eigenspace associated with the \emph{largest} singular value, an operation often achieved much more efficiently. 

\paragraph{The tradeoff}
While the per-iteration complexity of FW can be appealing, the downside is that the convergence rate is usually slow, compared against projected gradient methods \citep{jaggi2013revisiting,lacoste2015global,freund2016new}; in particular, without special assumptions beyond smoothness, away steps, or use of line search, the best known convergence rate of  \FW~is $O(1/k)$. 



In this paper, we explore this fundamental rate as it applies to the continuous \FWflow~and its various discretizations.
\section{Continuous time Frank-Wolfe}
The use of continuous time analysis is not new \citep{jacimovic1999continuous, su2016differential, shi2019acceleration}; importantly, continuous-time analysis of usual gradient method and its accelerated rate was shown to have the same convergence rate as its discretized version \citep{su2016differential}. 
The same is not true for the Frank-Wolfe method. 

\subsection{Continuous Frank-Wolfe flow rate}
\begin{proposition}[Continuous flow rate]
\label{prop:fwflow-rate}
Suppose that $\gamma(t) = \frac{c}{c+t}$, for some constant $c \geq 1$. Then the flow rate of \FWflow~ has an upper bound of 
\begin{equation}
\frac{f(x(t))-f^*}{f(x(0))-f^*} \leq  \left(\frac{c}{c+t}\right)^c =  O\left(\frac{1}{t^c}\right).
\label{eq:cont-upper-rate}
\end{equation}
\end{proposition}
\begin{proof}
\footnote{Much of this proof is standard analysis for continuous time Frank-Wolfe, and is also presented in \cite{jacimovic1999continuous}. }
Note that by construction of $\nabla f(x)^Ts = \displaystyle\min_{y\in\mD}\, \nabla f(x)^Ty$, and since $f$ is convex, 
\[
f(x) - f(x^*) \leq \nabla f(x)^T(x-x^*) \leq \nabla f(x)^T(x-s).
\] 
Quantifying the objective value error as $\mE(t) = f(x(t))-f^*$ 
(where $f^* = \min_{x\in \mD} f(x)$ is attainable) then 
\begin{eqnarray*}
\dot \mE(t) &=& \nabla f(x(t))^T\dot x(t) \\
&\overset{\FWflow}{=}& \gamma(t) \nabla f(x(t))^T(s(t)-x(t)).
\end{eqnarray*}
Therefore, 
\begin{eqnarray*}
\dot \mE(t) &=& -\gamma(t) \underbrace{\nabla f(x(t))^T(x(t)-s(t))}_{\geq f(x)-f(x^*)}\\
&\leq& -\gamma(t) \mE(t) 
\end{eqnarray*}
giving an upper rate of
\[
\mE(t) \leq \mE(0) e^{-\int_0^t\gamma(\tau)d\tau}.
\]
In particular, picking the ``usual step size sequence"  $\gamma(t) = \tfrac{c}{t+c}$ gives the proposed rate \eqref{eq:cont-upper-rate}.
\end{proof}

Note that this rate  is \emph{arbitrarily fast}, as long as we keep increasing $c$.
This is in stark contrast to the usual convergence rate of the Frank-Wolfe method, which in general \emph{cannot} improve beyond $O(1/k)$ for \emph{any} $c$. Later, we will see that this is true for \emph{all} Runge-Kutta type methods (of which \FW~is one type).
Figure \ref{fig:limit_continuous} shows this continuous rate as the limiting behavior of \FW, where the discretization steps $\Delta\to 0$. As proposed, the limiting behavior indeed improves with growing $c$; however, Figure \ref{fig:limit_continuous} also alludes that this idealized rate does not occur in local convergences of actual discretizations. 
\begin{figure*}
    \centering
    \includegraphics[width=5.5in,trim={2.5cm 8cm 2.5cm 0},clip]{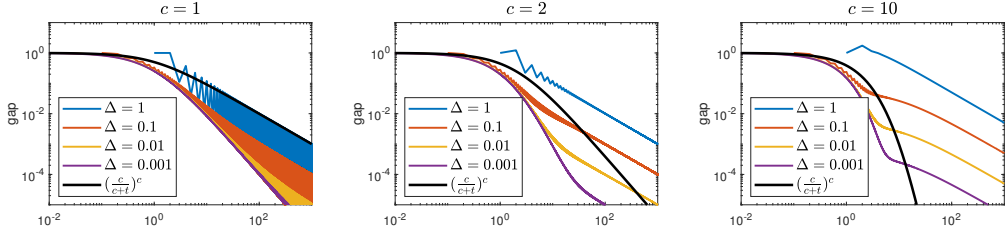}
    \caption{\textbf{Continuous vs discrete.} A comparison of the numerical error vs 
    compared with derived rate, on \eqref{eq:triangle_example}. The black curve shows the upper bound on the flow rate, compared against simulated method rates for smaller discretization units. The two-stage behavior of the curves is intriguing, as it seems there is a fundamental point where discretization error takes over, and forces the $O(1/k)$ rate to manifest.}
    \label{fig:limit_continuous}
\end{figure*}

\subsection{Continuous time Frank Wolfe does not zig-zag}

We first provide Figure \ref{fig:triangle_zigzag}  as an example of zig-zagging behavior over a toy problem 
\begin{equation}
\min_{\bx\in \mD}\;  \tfrac{1}{2}\|\bx-\bx^*\|_2^2,\quad  \mD:=\co \{(-1,0),(1,0),(0,1)\}
\label{eq:triangle_example}
\end{equation}
and $\co(\mS)$ is the convex hull of the set $\mS$.
From the examples in Figure \ref{fig:triangle_zigzag}, we are tempted to conclude that, indeed, zigzagging is a discretization phenomenon, and not a characteristic in the limiting  \FWflow.

\begin{figure}
    \centering
        \includegraphics[height=2.5in,trim={1.5cm 1cm 1.5cm .5cm},clip]{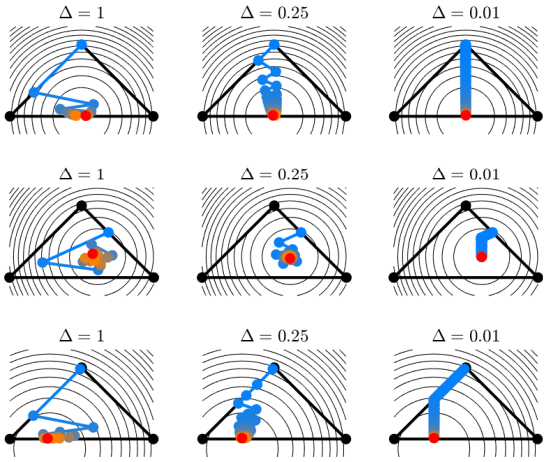}\\
       \includegraphics[height=2.5in,trim={0cm 0cm 1.5cm 0cm},clip]{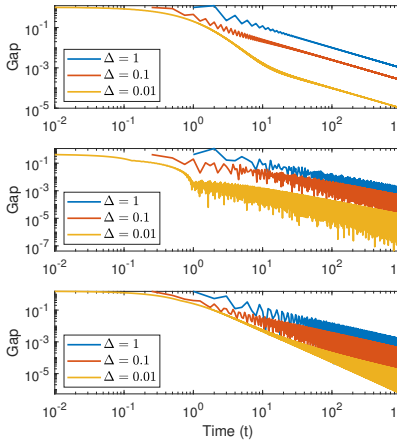}
    \caption{\textbf{Zig-zagging behavior on \eqref{eq:triangle_example}.} Here we show three random choices of $\bx^{(0)}$ and $\bx^*$, leading to three different but all zig-zaggy trajectories. Three discretizations are given; $\Delta = 1$ corresponds to \FW; visually, $\Delta=0.01$ approaches \FWflow.}
    \label{fig:triangle_zigzag}
\end{figure}

Figure \ref{fig:continuous_no_zigzag} quantifies this notion more concretely. We first propose to measure ``zig-zagging energy" by averaging the deviation of each iterate's direction across $k$-step directions, for $k = 1,...,W$, for some measurement window $W$:
\begin{multline*}
\mE_{\mathrm{zigzag}}(\bx^{(k+1)},...,\bx^{(k+W)}) =\\ \frac{1}{W-1}\sum_{i=k+1}^{k+W-1} \Big\|\underbrace{\left(I-\frac{1}{\|{\bar\bd}^{(k)}\|_2}\bar\bd^{(k)}({\bar\bd}^{(k)})^T\right)}_{\mb Q}\bd^{(i)}\Big\|_2,
\end{multline*}
where $\bd^{(i)} = \bx^{(i+1)}-\bx^{(i)}$ is the current iterate direction and $\bar\bd^{(k)} = \bx^{(k+W)}-\bx^{(k)}$ a ``smoothed" direction. The projection operator $\mb Q$ removes the component of the current direction  in the direction of the smoothed direction, and we measure this ``average deviation energy."
We divide the trajectory into these window blocks, and report the average of these measurements $\mE_{\mathrm{zigzag}}$ over $T=100$ time steps (total iteration = $T/\Delta$). 
Figure \ref{fig:continuous_no_zigzag} (top table) exactly shows this behavior, where the problem is sparse constrained logistic regression minimization over several machine learning classification datasets~\citep{guyon2004result} (Sensing (ours), Gisette \footnote{Full dataset available at \url{https://archive.ics.uci.edu/ml/datasets/Gisette}. We use a subsampling, as given in 
\url{https://github.com/cyrillewcombettes/boostfw}.
} and Madelon \footnote{Dataset: \url{https://archive.ics.uci.edu/ml/datasets/madelon}})
are shown in Fig. \ref{fig:continuous_no_zigzag}.




From this experiment, we see that any Euler discretization of \FWflow~will always zig-zag, in that the directions often alternate. But, by measuring the deviation across a windowed average, we see two things: first, the deviations converge to 0 at a rate seemingly linear in $\Delta$, suggesting that in the limit as $\Delta \to 0$, the trajectory is smooth. Furthermore, since these numbers are more-or-less robust to windowing size, it suggests that the smoothness of the continuous flow is on a macro level.

\begin{figure*}
    \centering
    \begin{tabular}{lr}
    \begin{minipage}{.2\textwidth}
    \includegraphics[width=1.5in]{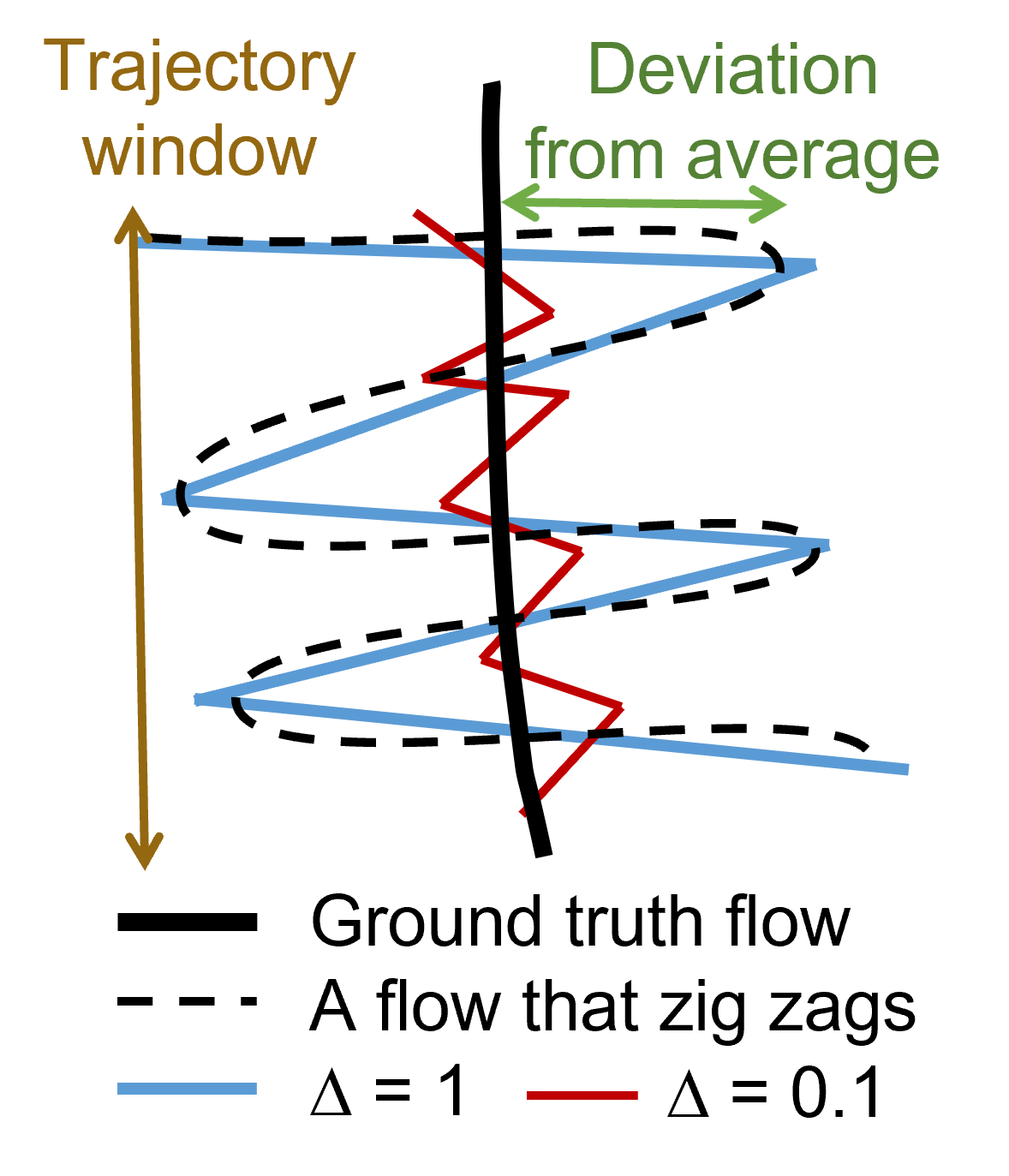}
    \end{minipage}
    &
    \begin{minipage} {.65\textwidth}
    \begin{tabular}{c}
    
    \begin{minipage}{.8\textwidth}
    \begin{tabular}[t]{c|c|c|c}
    \hline
    Test set&$\Delta = 1$&$\Delta = 0.1$ & $\Delta = 0.01$\\
    \hline
    Sensing & 105.90 / 140.29 & 10.49 / 13.86 & 1.05 / 1.39\\
    Madelon &0.11 / 0.23 &0.021 / 0.028&0.0021 / 0.0028\\
    Gisette &1.08 / 1.74&0.21 / 0.28&0.021 / 0.028\\
    \hline
    \end{tabular}
    \begin{center}
        Zigzagging in continuous flow
    \end{center}
    \end{minipage}
    \\
     \\
    \begin{minipage}{.8\textwidth}
    \begin{tabular}[t]{c|c|c|c}
    \hline
    Test set& FW & FW-MID & FW-RK4\\
    \hline
    Sensing & 105.90 / 140.29 & 0.57 / 1.0018 & 0.015 / 0.033\\
    Madelon & 0.031 / 0.040 & 0.025 / 0.029 &0.025 / 0.029\\
    Gisette & 0.30 / 0.40 & 0.25 / 0.11 & 0.22 / 0.20\\
    \hline
    \end{tabular}
    \begin{center}
        Zigzagging in multistep methods
    \end{center}
    \end{minipage}
    \end{tabular}
    
    \end{minipage}
    
    \end{tabular}
    
    \caption{\textbf{Zig-zagging on real datasets.} 
    Average deviation of different discretizations of \FWflow. Top table uses different $\Delta$s and uses the vanilla Euler's discretization (FW). Bottom uses $\Delta = 1$ and different multistep methods.  
    The two numbers in each box correspond to window sizes 5 / 20.
    }
    \label{fig:continuous_no_zigzag}
\end{figure*}




\section{Runge-Kutta multistep methods}
\label{sec:rkmethod}
\subsection{The generalized Runge-Kutta family}
We now look into multistep methods that better imitate the continuous flow by reducing discretization error. 
Observe that the standard FW algorithm is equivalent to the discretization of \FWflow~by Forward Euler's method with step size $\Delta = 1$. It is well known that the discretization error associated with this scheme is $O(\Delta^p)$ with $p = 1$, e.g. it is a method of order 1.

We now consider Runge-Kutta (RK) methods, a generalized class of higher order methods ($p \geq 1$).
These methods are fully parametrized by some choice of $A\in \R^{q\times q}$, $\beta\in \R^q$, and $\omega\in \R^q$ and at step $k$ can be expressed as (for $i = 1,...,q$)
\begin{equation}
\begin{array}{lcl}
\xi_i &=& \displaystyle\dot x\big(k+\omega_i ,\;  \bx^{(k)}+  \sum_{j=1}^q A_{ij} \xi_j\big),\\
\bx^{(k+1)} &=& \bx^{(k)}+ \sum_{i=1}^q \beta_i \xi_i.
\end{array}
\label{eq:general_discrete}
  \end{equation}
  For consistency,  $\sum_i \beta_i = 1$, and to maintain explicit implementations,  $A$ is always strictly lower triangular. As a starting point, $\omega_1 = 0$.
  Specifically, we refer to the iteration scheme in \eqref{eq:general_discrete} as a \emph{$q$-stage RK discretization method}.
  
  \begin{proposition}
For a given $q$-stage RK method defined by $A$, $\beta$, and $\omega$, for each given $k\geq 1$, define
\[
 \bar \gamma_i^{(k)} = \frac{c}{c+k+\omega_i}, \qquad 
 \Gamma^{(k)} = \diag({\bar \gamma^{(k)}}_i),
 \]
 \[
 \mathbf P^{(k)} = \Gamma^{(k)} (I+A^T\Gamma^{(k)})^{-1}, \qquad \mathbf z^{(k)} = q \mathbf P^{(k)}\beta.
\]
Then if $0\leq \mathbf z^{(k)}\leq 1$ for all $k\geq 1$, then 
\[
\bx^{(0)}\in \mD\Rightarrow \bx^{(k)}\in \mD, \quad \forall k\geq 1.
\]
  \end{proposition}
  \begin{proof}
  For a given $k$, construct additionally
  \[
\mathbf Z = \begin{bmatrix} \xi_1 & \xi_2 & \cdots & \xi_q \end{bmatrix},
\]
\[
\bar{\mathbf X} = \begin{bmatrix} \bar \bx_1 & \bar \bx_2 & \cdots & \bar \bx_q \end{bmatrix},\quad
\bar{\mathbf S} = \begin{bmatrix} \bar \bs_1 & \bar \bs_2 & \cdots & \bar \bs_q \end{bmatrix}.
\]
where
      \begin{eqnarray*}
 {\bar \bx}_i &=& \bx^{(k)}+  \sum_{j=1}^q A_{ij} \xi_j, \\
 {\bar \bs}_i &=& \lmo({\bar \bx}_i).
\end{eqnarray*}
  Then we can rewrite \eqref{eq:general_discrete} as
  \begin{eqnarray*}
  \mathbf Z &=& (\bar {\mathbf S} - \bar{\mathbf X})\Gamma = (\bar {\mathbf S} - \bx^{(k)}\mathbf 1^T - \mathbf ZA^T)\Gamma \\
  &=& (\bar {\mathbf S} - \bx^{(k)}\mathbf 1^T )\mathbf P
  \end{eqnarray*}
  for shorthand  $\mathbf P = \mathbf P^{(k)}$ and $\Gamma=\Gamma^{(k)}$. 
 Then
 \begin{eqnarray*}
\bx^{(k+1)} 
&=& \bx^{(k)}(1-\mathbf 1^T \mathbf P\beta) + \bar {\mathbf S} \mathbf P\beta\\
&=&\frac{1}{q}\sum_{i=1}^q  \underbrace{(1-z_i)\bx^{(k)} + z_i\bar \bs_i }_{\hat\xi_i}
  \end{eqnarray*}
  where $z_i$ is the $i$th element of $\bz^{(k)}$, and $\beta = (\beta_1,...,\beta_q)$. 
Then if $0\leq z_i\leq 1$, then $\hat \xi_i$ is a convex combination of $\bx^{(k)}$ and $\bar \bs_i$, and $\hat \xi_i\in \mD$ if $\bx^{(k)}\in \mD$. Moreover, $\bx^{(k+1)}$ is an average of $\hat\xi_i$, and thus $\bx^{(k+1)}\in \mD$. Thus we have recursively shown that $\bx^{(k)}\in \mD$ for all $k$.
  \end{proof}
The condition  $0\leq \bz^{(k)}\leq 1$ can be checked explicitly and is true of almost all RK methods with a notable exception of the midpoint method, where $\bz^{(k)}_i<0$ is possible. The implication is that for most \emph{other} FW-RK methods, the iterates \emph{maintain feasibility}; this is a defining characteristic of the vanilla FW methods.
A full list of the RK methods used in our experiments is described in the Appendix.

\paragraph{Error comparison.}
The \emph{total accumulation error (TAE)}, the distance between flow and discretization trajectory, is described as
\[
\epsilon_k = \|\bx^{(k)} - \bx^{(0)} - \int_0^{k\Delta} \dot \bx(t) dt \|.
\]
The method is of order $p$ if its TAE is $O(\Delta^p)$.
For example, the vanilla Frank-Wolfe (FW) has TAE of order $p =1$, the midpoint method (FW-MD) order $p=2$, and an RK-44 discretization (FW-RK4) with order $p = 4$. 
Figure \ref{fig:triangle_momentum} (top row) compares these three implementations on problem \eqref{eq:triangle_example}, and shows their rate of convergence. The closeness of the new curves with the continuous flow is apparent; however, while multistep methods are converging faster than vanilla FW, the rate does not seem to change.


\paragraph{No zig-zagging!}
One thing that \emph{is} visually apparent in Figure \ref{fig:triangle_momentum}  is that higher order  multistep methods establish better search directions. 
Additionally, we numerically quantify less zig-zagging behavior (lower table in Figure \ref{fig:continuous_no_zigzag}). This is still good news, as there are still several key advantages to such an improvement: namely, better uses of momentum and line search.

\section{RK convergence behavior}

\subsection{FW-RK is as good as FW.}
We first establish that using a generalized Runge-Kutta method cannot \emph{hurt} convergence, as compared to the usual Frank-Wolfe method.
\begin{proposition} 
All Runge-Kutta methods converge at worst with rate  $f(\bx^{(k)})-f(\bx^*)\leq O(1/k)$.
\label{prop:rungekutta-positive}
\end{proposition}
The proof is in Appendix. 
That is to say, an RK method cannot be an order \emph{slower} than vanilla FW.
\footnote{It should be noted, however, that the convergence rate in terms of $k$ does not account for the extra factor of $q$ gradient calls needed for a $q$-stage method. While this may be burdensome, it does not increase the \emph{order} of convergence rate.}

\subsection{FW-RK cannot be better than  FW}
Recall that \FWflow~achieves a rate of $O(1/t^c)$ rate, and is \emph{arbitrarily} fast as $c\to + \infty$. This is also verified numerically in Figure \ref{fig:limit_continuous}; larger $c$ provides a sharper local convergence rate. 
It is hence tempting to think that increasing $c$ can help FW methods in general, and in particular by adapting a higher order multistep method, we can overcome the problems caused by discretization errors. 

To see this in terms of bounds, we can model a discretization method as having two error terms:
\[
\mE_c^{(k)} =  \left(\frac{c}{c+\Delta k}\right)^c, \qquad \mE_d^{(k)} = \mE_d := \frac{1}{\Delta^p}, 
\]
and in general, our guarantees only say that 
\[
\frac{f(\bx^{(k)})-f(x^*)}{f(\bx^{(0)})-f^*} \leq \max\{\mE_c^{(k)},\mE_d\}.
\]
That is to say, \emph{continuous time analysis does not guarantee any convergence when $\mE_c^{(k)} < \mE_d$}. This is a cautionary tail, in that while continuous-time analysis may be mathematically beautiful and offer many insights, \emph{it may not have a 1-1 correspondence with any implementable rate.}


\paragraph{Lower bound derivation: a toy problem.}
Let us now consider a simple bounded optimization problem over scalar variables $\bx \in \mathbb R$:
\begin{equation}
        \min_{\bx} \; f(\bx)\quad \mathrm{s.~t.} \; -1\leq \bx \leq 1
        \label{eq:toyproblem}
\end{equation}
where 
\[
f(\bx) = 
\begin{cases}
\bx^2/2 & \text{if } |\bx| < \varepsilon\\
\varepsilon \bx-\varepsilon^2/2 & \text{if } \bx \geq \varepsilon\\
-\varepsilon \bx - \varepsilon^2/2 & \text{if } \bx \leq -\varepsilon\\
\end{cases}
\]
which is a scaled version of the Huber norm applied to scalars. By design, no matter how small $\varepsilon$ is, $f$ is $L$-smooth and 1-Lipschitz. Additionally,  $\lmo_{[-1,1]}(\bx) = -\sign(\bx)$. The flow corresponding to \eqref{eq:toyproblem} can be summarized as 
\[
\dot x(t) = -\gamma(t)(\sign(x(t))+x(t))
\]
or, in terms of $u(t) = |x(t)|$,
$\dot u = -\gamma(u+1)$.
Taking our usual $\gamma(t) = c/(c+t)$ and solving the ODE using separation of variables gets
\[
u(t) = (u_0+1)\left(\frac{c}{c+t}\right)^c-1
\]
which gives a $O(1/t^c)$ rate of convergence for $u+1$, and a \emph{finite-time} convergence for $u(t)$. 

This is interesting and surprising, but alas, not all that useful, as numerical experiments do not show that this fast rate can be leveraged. 

\begin{proposition} 
Assuming that $0 <q\mb P^{(k)} \beta < 1$ for all $k$.
Start with $\bx^{(0)} = 1$.
Suppose 
the choice of $\beta\in \R^p$ is not ``cancellable"; that is, there exist no partition $S_1\cup S_2 =\{1,...,p\}$ where 
\[
\sum_{i\in S_1}\beta_i - \sum_{j\in S_2} \beta_j = 0.
\]
Then regardless of the order $p$ and choice of $A$, $\beta$, and $\omega$, as long as $\sum_i \beta_i = 1$, then 
\[
\sup_{k'>k} |\bx^{(k)}| =\Omega(1/k).
\]
That is, the tightest upper bound is $O(1/k)$.
\label{prop:rungekutta-negative}
\end{proposition}

The proof is in Appendix. 
The assumption of a ``non-cancellable" choice of $\beta_i$ may seem strange, but in fact it is true for most of the higher order  Runge-Kutta methods. More importantly, the assumption doesn't matter in practice; even if we force $\beta_i$'s to be all equal, our numerical experiments do not show much performance difference in this toy problem.  (Translation: do not design your Runge Kutta method for the $\beta$'s to be cancel-able in hopes of achieving a better rate!)

Proposition \ref{prop:rungekutta-positive} implies a $\Omega(1/k)$ bound on $|\bx_k|$. To extend it to an $\Omega(1/k)$ bound on $f(\bx_k)-f^*$, note that whenever $|\bx| \geq \varepsilon$,
\[
\frac{f(\bx_k)-f^*}{f(\bx_0)-f^*} \geq \frac{|\bx_k|}{2|\bx_0|}
\]
for $\varepsilon$ \emph{arbitrarily small}. \begin{corollary}
The worst best case bound for FW-RK, for any RK method, is of order $O(1/k)$.
\end{corollary}

\section{A better search direction}
We have now given disappointing proof that multistep methods cannot give acceleration over the general class of problems \eqref{eq:fw-main}.
However, as previously discussed, it still seems that the search \emph{direction} is of better quality, when the zig-zagging is removed; this is evidenced by the second table in Fig. \ref{fig:continuous_no_zigzag}. We leverage this in two ways. 
\paragraph{A more aggressive line search.}
Suppose that a multistep method, as presented in the previous section, updates 
\[
\bx^{(k+1)} = \bx^{(k)} + \gamma^{(k)} \bd^{(k)}
\]
First, we consider more aggressive line searches, e.g. replacing $\gamma^{(k)}$ with $\max\{\frac{2}{2+k},\bar \gamma\}$ and
\[
\bar\gamma = \max_{0\leq \gamma\leq 1}\{\gamma : f(\bx^{(k)} + \gamma^{(k)} \bd^{(k)}) \leq f(\bx^{(k)})\}.
\]
Note that this is not the typical line search as in \citep{lacoste2015global}, which forces $\gamma^{(k)}$ to be upper bounded by $O(1/k)$--we are hoping for more aggressive, not less, step sizes. 

\paragraph{Better use of momentum.} Second, we consider the benefit of adding a momentum term. Intuitively,  momentum acts as a heavy ball, and does not behave well in the presence of zig-zagging; eliminating zig-zags should thus allow momentum to provide more benefit.
Specifically, we follow the scheme presented in \citet{li2020does} which generalizes the 3-variable Nesterov acceleration  \citep{nesterov2003introductory} from gradient descent to Frank-Wolfe. 
In the adaptation for Frank-Wolfe, each time the LMO is called, we update a momentum term for the gradient, as shown below.
\begin{eqnarray*}
\by^{(k)} &=& (1-\gamma_k)\bx^{(k)}+\gamma_k \bv^{(k)},\\
\bz^{(k+1)} &=& (1-\gamma_k)\bz^{(k)} + \gamma_k \nabla f(\by^{(k)}),\\
\bv^{(k+1)} &=& \lmo_\mD(\bz^{(k+1)}),\nonumber\\
\bx^{(k+1)} &=& (1-\gamma_k) \bx^{(k)} +  \gamma_k \bv^{(k+1)}\\
\end{eqnarray*}
\newline\noindent Although more memory is required for the extra variables, no extra gradient or LMO calls are required. 

Fig. \ref{fig:triangle_momentum}, rows 2 and 3, illustrate the benefits of multistep methods for line search (row 2) and momentum (row 3), over the toy problem \eqref{eq:triangle_example}.

\begin{figure*}[!htb]
   \centering
  \includegraphics[height=1in,trim={2cm 0 2cm 0},clip]{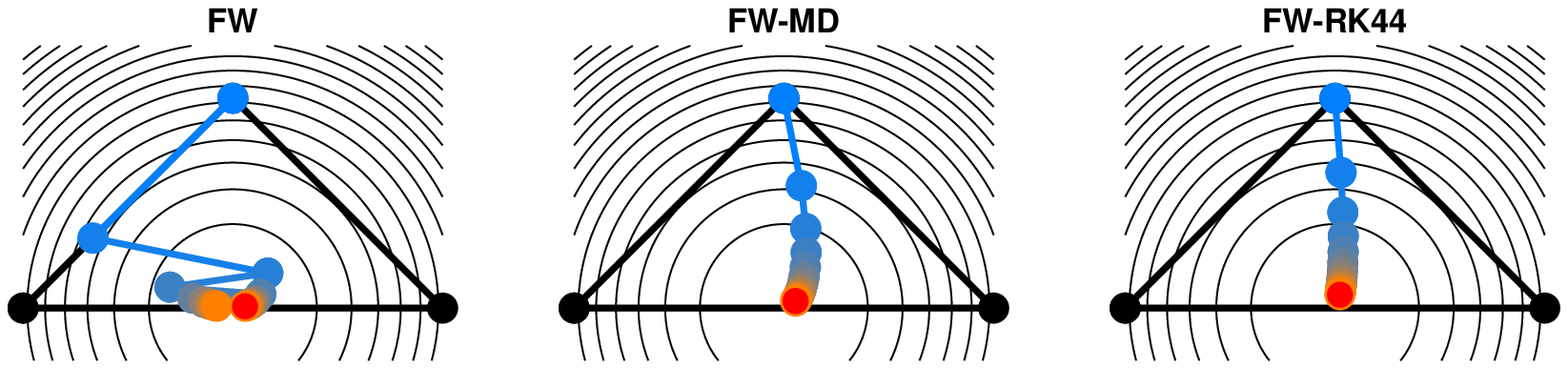}
  \includegraphics[height=1in]{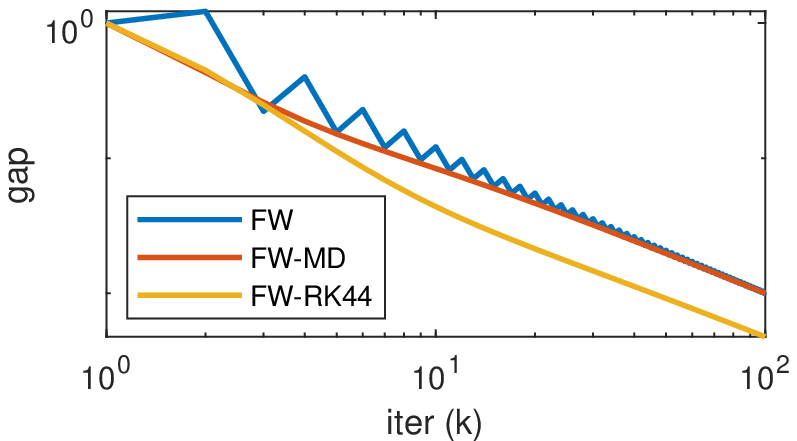}\\

   \centering
  \includegraphics[height=1in,trim={2cm 0 2cm 0},clip]{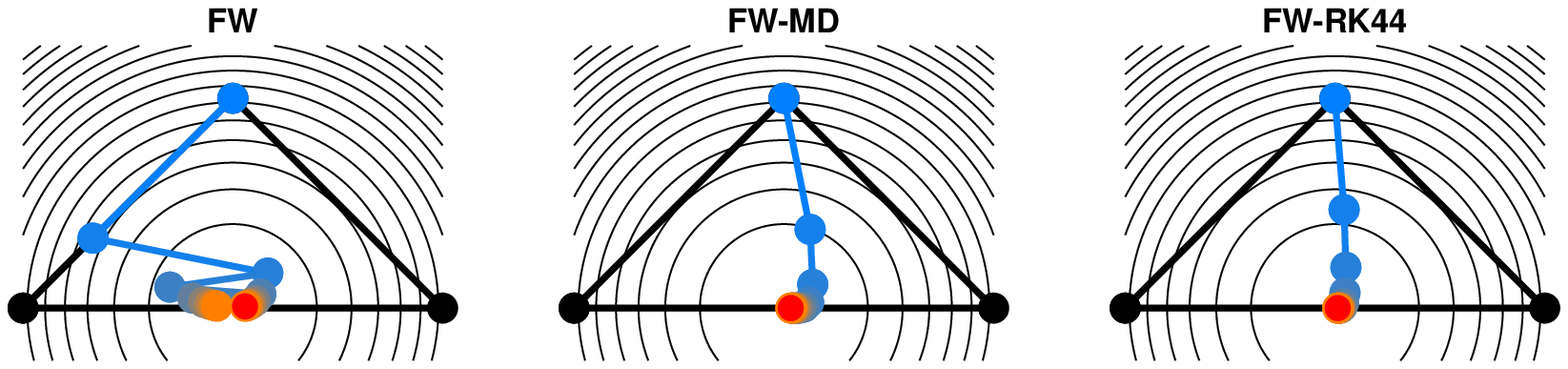}
  \includegraphics[height=1in]{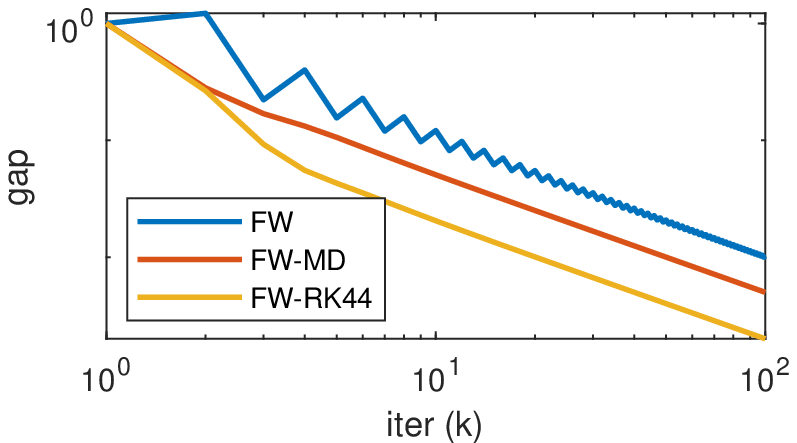}

  \includegraphics[height=1in,trim={2cm 0 2cm 0},clip]{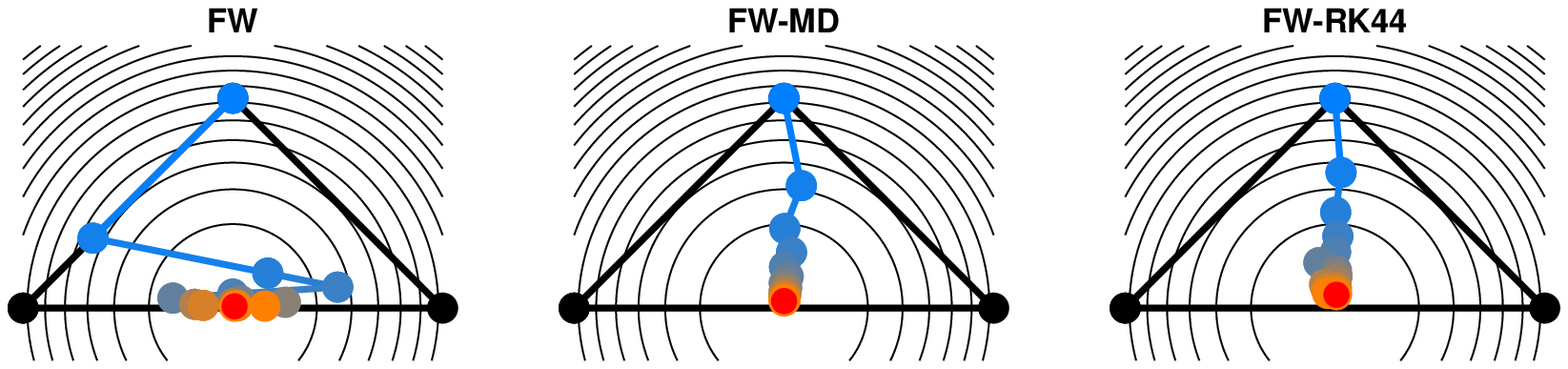}
  \includegraphics[height=1in]{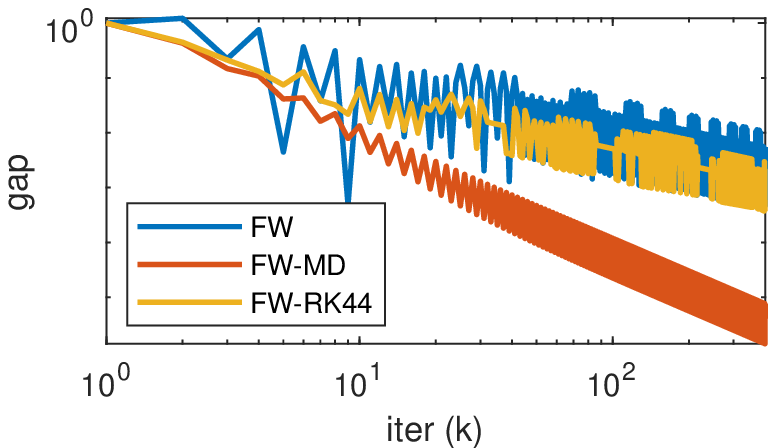}
 \caption{\textbf{Triangle toy problem.} Top: Straight implementions. Middle: Line search. Bottom: Momentum acceleration.}
 \label{fig:triangle_momentum}
\end{figure*}

\section{Numerical experiments}
In this section we evaluate the benefit of our multistep Frank-Wolfe methods on three tasks:
\begin{enumerate}
    \item simulated compressed sensing, with a Gaussian i.i.d. sensing matrix and a noisy observation of a sparse ground truth; 
    \item sparse logistic regression over the Gisette \citep{guyon2004result} dataset; 
    \item and low-rank matrix completion over the MovieLens 1M dataset.
\end{enumerate}
We implement the algorithms in MATLAB and Python and make our code available on Github.
\footnote{\url{https://github.com/Mokhwalee/Continuous-Time-Frank-Wolfe-Does-not-Zig-Zag}}
\paragraph{Simulated compressed sensing}
This task minimizes a quadratic  function with a $\ell_1$- norm constraint. Given $G\in \R^{n\times m}$ with entries i.i.d. Gaussian, we generate $h = Gx_0  + z$ where $x_0$ is a sparse ground truth vector with 10\% nonzeros, and $z_i\sim \mN(0,0.05)$. We solve the following convex problem
\begin{equation}
\min_{x\in \R^n} \quad  \tfrac{1}{2}\|Gx-h\|_2^2\qquad
\mathrm{subject~to} \quad  \|x\|_1\leq\alpha.
\end{equation}
Figure \ref{fig:quadratic} evaluates the performance of the multistep Frank-Wolfe methods, boosted by line search, for a problem with  $m = 500$, $n = 100$, and $\alpha = 1000$.
\begin{figure}[!htb]
  \includegraphics[width=\linewidth]{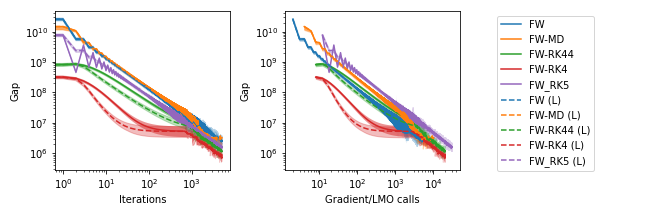}
  \caption{\textbf{Compressed sensing.} $500$ samples, $100$ features, 10\% sparsity ground truth, $\alpha = 1000$. L = line search. Performed over 10 trials.}
  \label{fig:quadratic}
\end{figure}



\paragraph{Sparse logistic regression} We now solve a similar problem with logistic loss
\begin{equation}
    \min_{x\in \R^n} \;  \frac{1}{m}\sum_{i=1}^{m}\log(1+\exp(-y_iz_i^Tx))\quad
    \mathrm{s.t.} \;  \|x\|_1\leq\alpha
\end{equation}
where the problem features $z_i$ and labels $y_i\in \{-1,1\}$ are coming from the Gisette  task
\cite{guyon2004result}. The purpose of this task is to recognize grayscale images of two confusable handwritten digits: 4 and 9. 
Again, we see that the multistep method can boost the acceleration given by adding momentum. 

\begin{figure}[!htb]
  \includegraphics[width=\linewidth]{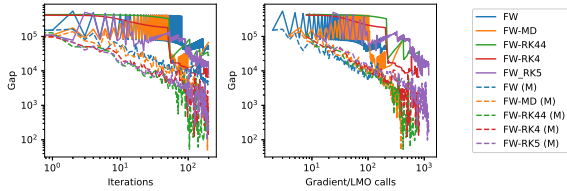}
  \caption{\textbf{Gisette.} The dataset has 2000 samples and 5000 features. $\alpha = 250$. M = momentum. }
  \label{fig:logistic}
\end{figure}



\paragraph{Nuclear-Norm Constrained Huber Regression}
Finally, we consider the low-rank matrix factorization problem over the Movielens 100K dataset \citep{harper2015movielens}
\footnote{MovieLens dataset is available at \url{https://grouplens.org/datasets/movielens/100k/}}.
This dataset contains 100,000 ratings (1-5) from 943 users on 1682 movies. Define $O$ as the set of observed indices, e.g. $(i,j)\in O$ if user $i$ rated movie $j$. We normalize the ratings so that $R = R_0-3$, where $R_0$ are the raw values given from the dataset. The goal is then to solve
\begin{equation}
    \displaystyle\min_{\bX\in \R^{n\times m}}   \sum_{i,j \in O } H(R_{i,j}-\bX_{i,j})\quad
    \mathrm{s.t.} \quad  \|\bX\|_*\leq\alpha,
    \label{eq:matfact}
    \end{equation}
    where 
    \[
    H(\xi) = \begin{cases}
    \frac{1}{2} \xi^2, & |\xi| \leq \rho\\
    \rho |\xi-\rho| + \frac{1}{2}\rho^2, & |\xi| > \rho.
    \end{cases}
    \]
    is the Huber norm.
Figure \ref{fig:matfact} gives a comparison of the different methods for solving \eqref{eq:matfact}, with and without momentum.
\begin{figure}[!htb]
  \includegraphics[width=\linewidth]{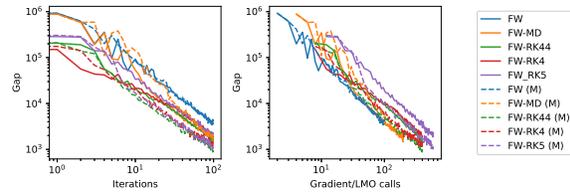}
  \caption{\textbf{Movielens.} The dataset contains 100,000 ratings (1-5) from 943 users on 1682 movies. We normalize the dataset to range (-2,2), and solve \eqref{eq:matfact} with $\alpha = 1000$, $\rho = 10$. M = momentum.}
  \label{fig:matfact}
\end{figure}



\section{Discussion}
The goal of this work is to put forth a thorough analysis of multistep methods over the flow system captured by the Frank-Wolfe method. Our results are numerically promising but theoretically disheartening; on the one hand, because the flow system does not zig-zag, better discretization methods do offer more stable search directions and momentum performance. On the other hand, because truncation error ultimately dominates convergence error, better multistep methods in general cannot improve local convergence rates. This   discovery  has implications not just for Frank-Wolfe methods, but for \emph{any} multistep optimization methods to beat the original convergence rates, and is a cautionary tale for assuming too much from a beautiful flow rate.



\bibliography{main.bbl}
\bibliographystyle{icml2022}

\newpage
\appendix
\onecolumn

\section{Runge Kutta methods}


\begin{itemize}
    \item Midpoint method
    \[
    A = \begin{bmatrix}
    0 & 0 \\ 1/2 & 0
    \end{bmatrix},
    \qquad
    \beta = \begin{bmatrix}
    0 \\1 
    \end{bmatrix},
    \qquad
    \omega = \begin{bmatrix}
    0\\1/2
    \end{bmatrix},\qquad
    \bz^{(1)} \approx \begin{bmatrix}
    -0.3810\\1.1429
    \end{bmatrix},\qquad
    \bz^{(2)} \approx \begin{bmatrix}
    -0.2222\\0.8889
    \end{bmatrix}
\]
     \item Runge Kutta 4th Order Tableau (44)
    \[
    A = \begin{bmatrix}
   0    & 0    & 0 & 0 \\
     1/2  & 0    & 0 & 0 \\
     0  & 1/2  & 0 & 0 \\
    0    & 0    & 1 & 0 
    \end{bmatrix},
    \qquad
    \beta = \begin{bmatrix}
  1/6   \\ 1/3  \\ 1/3 \\ 1/6
    \end{bmatrix},
    \qquad
    \omega = \begin{bmatrix}
    0   \\
        1/2 \\
        1/2 \\
        1  
    \end{bmatrix},\qquad
    \bz^{(1)} \approx \begin{bmatrix}
    0.2449\\
    0.5986\\
    0.5714\\
    0.3333
    \end{bmatrix}
    \]
    
        \item Runge Kutta 3/8 Rule Tableau (4)
    \[
    A = \begin{bmatrix}
   0    & 0    & 0 & 0 \\
     1/3  & 0    & 0 & 0 \\
     -1/3  & 1  & 0 & 0 \\
     1    & -1    & 1 & 0 \\
    \end{bmatrix},
    \qquad
    \beta = \begin{bmatrix}
 1/8   \\3/8  \\ 3/8 \\ 1/8
    \end{bmatrix},
    \qquad
    \omega = \begin{bmatrix}
   0    \\
        1/3  \\
        2/3 \\
        1  
    \end{bmatrix},\qquad
    \bz^{(1)} \approx \begin{bmatrix}
    0.1758\\
    0.6409\\
    0.6818\\
    0.2500
    \end{bmatrix}
    \]
\item 
Runge Kutta 5 Tableau
\[
    A = \begin{bmatrix}
 0 & 0 & 0 & 0 & 0 & 0\\
 1/4 & 0 & 0 & 0 & 0 & 0\\
  1/8 & 1/8 & 0 & 0 & 0 & 0\\
 0 & -1/2 & 1 & 0 & 0 & 0\\
 3/16 & 0 & 0 & 9/16 & 0 & 0\\
 -3/7 & 2/7 & 12/7 & -12/7 & 8/7 & 0\\
    \end{bmatrix},
    \qquad
    \beta = \begin{bmatrix}
7/90 \\ 0 \\ 32/90 \\ 12/90 \\ 32/90 \\7/90
    \end{bmatrix},
    \qquad
    \omega = \begin{bmatrix}
    0 \\
        1/4 \\
        1/4 \\
        1/2 \\
        3/4\\
        1 
    \end{bmatrix},\qquad
    \bz^{(1)} \approx \begin{bmatrix}
      0.1821\\
    0.0068\\
    0.8416\\
    0.3657\\
    0.9956\\
    0.2333
    \end{bmatrix}
    \]
\end{itemize}
In all examples, $\|\bz^{(k)}\|_\infty$ monotonically decays with $k$.

\section{Positive Runge-Kutta convergence result}
\label{app:sec:positiveresults}

\begin{lemma} 
After one step, the generalized Runge-Kutta method satisfies
\[
h(\bx^{(k+1)})-h(\bx^{(k)}) \leq
-\gamma^{(k+1)} h(\bx^{(k)}) + D_4(\gamma^{(k+1)})^2\]

where $h(\bx) = f(\bx) - f(\bx^*)$ and


\[
D_4 =\frac{LD_2^2+2LD_2D_3+2D_3}{2},\quad D_2 = c_1 D, \quad D_3 = c_2c_1 D, \quad c_1 = qp_{\max}, \quad c_2 = q \max_{ij} |A_{ij}|, \quad D = \diam(\mD).
\]

\label{lem:rungekutta-positive-onestep}
\end{lemma}
\begin{proof}
For ease of notation, we write $\bx = \bx^{(k)}$ and $\bx^+ = \bx^{(k+1)}$. We will use $\gamma=\gamma^{(k)} = \tfrac{c}{c+k}$, and $\bar \gamma_i = \tfrac{c}{c+k+\omega_i}$.
Now consider the generalized RK method
\begin{eqnarray*}
\bar \bx_i &=& \bx + \sum_{j=1}^q A_{ij} \xi_j\\
\xi_i &=& \underbrace{\frac{c}{c+k+\omega_i}}_{\tilde \gamma_i} ( \bs_i - \bar {\bx}_i )\\
\bx^+&=&  \bx + \sum_{i=1}^q \beta_i \xi_i\\
\end{eqnarray*}
where $\bs_i = \lmo(\bar \bx_i)$.


Define $D = \diam(\mD)$. 
We use the notation from  section
\ref{sec:rkmethod}. 
Denote the 2,$\infty$-norm as
\[
\|A\|_{2,\infty} = \max_j \|a_j\|_2
\]
where $a_j$ is the $j$th column of $A$. Note that all the element-wise elements in 
\[
\mathbf P^{(k)} = \Gamma^{(k)}(I+A^T\Gamma^{(k)})^{-1}
\]
is a decaying function of $k$, and thus defining $p_{\max} = \|\mathbf P^{(1)}\|_{2,\infty}$
we see that
\[
\|\bar {\mathbf Z}\|_{2,\infty} = \|(\bar {\mathbf S} - \bx^{(k)}\mathbf 1)\mathbf P^{(k)}\|_{2,\infty} \leq qp_{\max} D.
\]

Therefore, since $\bar {\mathbf Z} = (\bar {\mathbf S}-\bar {\mathbf X})\Gamma$, and all the diagonal elements of $\Gamma$ are at most 1, 
\[
\|\bs_i-\bar \bx_i\|_2  \leq qp_{\max} D =: D_2
\]
and
\[
\|\bx-\bar \bx_i\|_2 = \|\sum_{j=1}^q A_{ij} \gamma_j (\bs_j-\bar \bx_j)\|_2 \leq q \max_{ij} |A_{ij}| \gamma D_2 =: D_3 \gamma.
\]

Then 
\begin{eqnarray*}
f(\bx^+)-f(\bx) &\leq&  \nabla f(\bx)^T(\bx^+-\bx) + \frac{L}{2}\|\bx^+-\bx\|_2^2\\
&=&  \sum_i \beta_i \tilde \gamma_i \nabla f(\bx)^T(\bs_i-\bar \bx_i) + \frac{L}{2}\underbrace{\|\sum_i \beta_i \tilde \gamma_i (\bs_i-\bar \bx_i)\|_2^2}_{\leq \gamma^2 D_2^2}\\
&=&  \sum_i \beta_i \tilde \gamma_i (\nabla f(\bx)-\nabla f(\bar \bx_i))^T(\bs_i-\bar \bx_i) +
 \sum_i \beta_i \tilde \gamma_i \underbrace{\nabla f(\bar \bx_i)^T(\bs_i-\bar \bx_i)}_{-\gap(\bar \bx_i)} +
 \frac{L\gamma^2D_2^2}{2}\\
 &\leq& \sum_i \beta_i\underbrace{ \tilde \gamma_i}_{\leq \gamma} \underbrace{\|\nabla f(\bx)-\nabla f(\bar \bx_i)\|_2}_{L\|\bx-\bar \bx_i\|_2=L\gamma D_3}\underbrace{\|\bs_i-\bar \bx_i\|_2}_{\leq D_2} - \sum_i \beta_i\tilde\gamma_i \gap(\bar \bx_i) + \frac{L\gamma^2 D_2^2}{2}\\
 &\leq &  -\sum_i \beta_i\tilde\gamma_i \gap(\bar \bx_i) + \frac{L\gamma^2 D_2^2}{2} + \frac{2L\gamma^2 D_2D_3}{2}\\ 
 &\leq& -\gamma^+ \sum_i\beta_i h(\bar \bx_i)  + \frac{L\gamma^2D_2(D_2+2D_3)}{2} 
\end{eqnarray*}
where $\gamma=\gamma_k$, and $\gamma^+=\gamma_{k+1}$. Now assume $f$ is also $L_2$-continuous, e.g. $|f(\bx_1)-f(\bx_2)|\leq L_2\|\bx_1-\bx_2\|_2$. Then, taking  $h(\bx) = f(\bx) -f(\bx^*)$,

\begin{eqnarray*}
h(\bx^+)-h(\bx) 
 &\leq& -\gamma^+ \sum_i\beta_i (h(\bar \bx_i)-h(\bx)) -\gamma^+ \underbrace{\sum_i\beta_i}_{=1} h(\bx) + \frac{L\gamma^2D_2(D_2+2D_3)}{2}\\
 &\leq & 
 \gamma \sum_i\beta_iL_2 \underbrace{\|\bar \bx_i-\bx\|_2}_{\leq \gamma D_3}-\gamma^+ h(\bx) + \frac{L\gamma^2D_2(D_2+2D_3)}{2}\\
 &\leq & -\gamma^+ h(\bx) + \frac{\gamma^2(LD_2^2+2LD_2D_3+2D_3)}{2}\\
 &\leq& -\gamma^+ h(\bx) + D_4(\gamma^+)^2
\end{eqnarray*}
where 
$D_4 =\frac{LD_2^2+2LD_2D_3+2D_3}{2}$and we use $2 \geq (\gamma/\gamma^+)^2$ for all $k \geq 1$.

\end{proof}

Proof of Prop. \ref{prop:rungekutta-positive}
\begin{proof}
After establishing Lemma \ref{lem:rungekutta-positive-onestep}, the rest of the proof is a recursive argument, almost identical to that in  \cite{jaggi2013revisiting}. 

At $k = 0$, we define $h_0 =\max\{ h(\bx^{(0)}), \frac{ D_4c^2}{c-1}\}$, 
and it is clear that $h(\bx^{(0)}) \leq  h_0$.

Now suppose that for some $k$, $h(\bx^{(k)}) \leq \frac{h_0}{k+1}$. Then
\begin{eqnarray*}
h(x_{k+1}) &\leq &  h(\bx_k) - \gamma_{k+1}h(\bx^{(k)}) + {D_4} \gamma_{k+1}^2\\
&\leq & \frac{h_0}{k+1}\cdot \frac{k+1}{c+k+1} + D_4 \frac{c^2}{(c+k+1)^2}\\
&=& \frac{h_0}{c+k+1} + D_4 \frac{c^2}{(c+k+1)^2}\\
&=& \left( h_0 + \frac{D_4c^2}{c+k+1}\right) \left(\frac{k+2}{c+k+1}\right) \frac{1}{k+2}
\\
&\leq& h_0\left( 1+\frac{c-1}{c+k+1}\right) \left(\frac{k+2}{c+k+1}\right) \frac{1}{k+2}\\
\\
&\leq& h_0\underbrace{\left( \frac{2c+ k }{c+k+1}\right) \left(\frac{k+2}{c+k+1}\right)}_{\leq 1} \frac{1}{k+2}.
\end{eqnarray*}
\end{proof}

\section{Negative Runge-Kutta convergence result}
\label{app:sec:negativeresults}

This section gives the proof for Proposition \ref{prop:rungekutta-negative}.

\begin{lemma}[$O(1/k)$ rate]\label{lem:o1krate}
Start with $\bx^{(0)} = 1$. Then consider the sequence defined by
\[
\bx^{(k+1)}= |\bx^{(k)} - \frac{c_k}{k}|
\]
where, no matter how large $k$ is, there exist some constant where  $C_1 < \max_{k'>k} c_{k'} $.
(That is, although $c_k$ can be anything, the smallest upper bound of $c_k$ does not decay.) Then
\[
\sup_{k'\geq k} |\bx^{(k')}| = \Omega(1/k).
\]
That is, the smallest upper bound of $|\bx^{(k)}|$ at least of order $1/k$.
\end{lemma}
\begin{proof}
We will show that the smallest upper bound of $|\bx^{(k)}|$ is larger than $C_1/(2k)$.

Proof by contradiction. 
Suppose that at some point $K$, for all $k \geq K$,  $|\bx^{(k)}| < C_1/(2k)$. Then from that point forward, 
\[
\sign(\bx^{(k)}-\frac{c_k}{k}) = -\sign(\bx^{(k)})
\]
and there exists some $k' > k$ where $c_{k'} > C_1$. Therefore, at that point,
\[
|\bx^{(k'+1)}| = \frac{c_{k'}}{k'}-|\bx^{(k')}| 
\geq \frac{C_1}{2k'}>\frac{C_1}{2(k'+1)}.
\]
This immediately establishes a contradiction.
\end{proof}

Now define the operator 
\[
T(\bx^{(k)}) = \bx^{(k+1)}-\bx^{(k)}
\]
and note that 
\[
|\bx^{(k+1)}| = |\bx^{(k)}+T(\bx^{(k)})| = | |\bx_k|+\sign(\bx^{(k)})T(\bx^{(k)})|.
\]
Thus, if we can show that there exist some $\epsilon$, agnostic to $k$ (but possibly related to Runge Kutta design parameters), and
\begin{equation}
\exists k'\geq k, \quad -\sign(\bx^{(k')})T(\bx^{(k')}) > \frac{\epsilon}{k'},\quad \forall k,
\label{eq:lemma-helper-1}
\end{equation}
 then based on the previous lemma, this shows $\sup_{k'>k}|\bx_{k'}| = \Omega (1/k)$ as the smallest possible upper bound.

\begin{lemma}
Assuming that $0 <q\mb P^{(k)} \beta < 1$ then there exists a finite point $\tilde k$ where for all $k > \tilde k$, 
\[
|\bx^{(k)}| \leq \frac{C_2}{k}
\]
for some $C_2 \geq 0$.
\end{lemma}
\begin{proof}

We again use the block matrix notation
\[
\bZ^{(k)} = (\bar \bS-\bx^{(k)}\mb 1^T) \Gamma^{(k)}(I+A^T\Gamma^{(k)})^{-1}
\]
where $\Gamma^{(k)} = \diag(\tilde \gamma_i^{(k)})$ and each element $\tilde \gamma_i^{(k)} \leq \gamma^{(k)}$.

First, note that by construction, since 
\[
\|\bar \bS-\bx^{(k)}\mb 1^T\|_{2,\infty} \leq D_4, \quad \|(I+A^T\Gamma^{(k)})^{-1}\|_2 \leq  \|(I+A^T\Gamma^{(0)})^{-1}\|_2
\]
are bounded above by constants, then 
\[
\|\bZ^{(k)}\|_\infty \leq \frac{c}{c+k} C_1
\]
for $C_1 = D_4\|(I+A^T\Gamma^{(0)})^{-1}\|_2 $.

First find constants $C_3 $, $C_4$, and $\bar k$ such that
\begin{equation}
\frac{C_3}{k} \leq \mb 1^T \mathbf P^{(k)} \beta \leq \frac{C_4}{k}, \quad \forall k>\bar k,
\label{eq:boundx_helper}
\end{equation}
and such constants always exist, since
by assumption, there exists some $a_{\min} > 0$, $a_{\max}<1$ and some $k'$ where
\[
a_{\min} <q\mb P^{(k')} \beta < a_{\max} \Rightarrow \frac{a_{\min}}{q \gamma_{\max}} \leq (I+A^T\Gamma^{(k')})^{-1} \beta \leq \frac{a_{\max}}{q \gamma_{\min}}
\]
where 
\[
\gamma_{\min} = \min_i \frac{c}{c+k'+\omega^{(k')}_i}, \qquad \gamma_{\max} = \frac{c}{c+k'}.
\]
Additionally, for all $k > c+1$,
\[
\frac{c}{2k}\leq \frac{c}{c+k+1} \leq \Gamma^{(k)}_{ii} \leq \frac{c}{c+k} \leq \frac{c}{k}.
\]
Therefore taking
\[
C_3 = \frac{ca_{\min}}{2q \gamma_{\max} }, \qquad C_4 = \frac{c a_{\max}}{q\gamma_{\min}}, \qquad \bar k = \max\{k',c+1\}
\]
satisfies \eqref{eq:boundx_helper}.

Now define
\[
C_2 = \max\{|\bx^{(1)}|,4cq C_1 \|A\|_\infty, 4C_3, 4C_4\}.
\]
We will now inductively show that $|\bx^{(k)}|\leq \frac{C_2}{k}$. From the definition of $C_2$, we have the base case for $k = 1$:
\[
|\bx^{(1)}| \leq \frac{|\bx^{(1)}|}{1} \leq \frac{C_2}{k}.
\]
Now assume that $|\bx^{(k)}|\leq \frac{C_2}{k}$. Recall that
\[
\bx^{(k+1)} = \bx^{(k)}(1-\mb 1^T \mathbf P^{(k)} \beta) + \bar \bS\mathbf P^{(k)} \beta, \qquad \bar \bS = [\bar \bs_1,...,\bar\bs_q], \qquad \bs_i = -\sign(\bar \bx_i)
\]
and we denote the composite mixing term $\bar \gamma^{(k)} = \mb 1^T\mb P^{(k)} \beta$. 
We now look at two cases separately.
\begin{itemize}
    \item Suppose first that $\bar \bS = -\sign(\bx^{(k)}\mb 1^T)$, e.g. $\sign(\bar\bx_i) = \sign(\bx^{(k)})$ for all $i$. Then 
    \[
    \bar \bS \mb P^{(k)} \beta = -\sign(\bx^{(k)})\bar \gamma_k,
    \]
    and 
    \begin{eqnarray*}
    |\bx^{(k+1)}| &=&  |\bx^{(k)}(1-\bar\gamma^{(k)}) + \bar \bS\mathbf P^{(k)} \beta|\\
    &=&  |\bx^{(k)}(1-\bar\gamma^{(k)}) -\sign(\bx^{(k)})\bar\gamma^{(k)}|\\
    &=&  |\underbrace{\sign(\bx^{(k)})\bx^{(k)}}_{|\bx^{(k)}|}(1-\bar\gamma^{(k)}) -\underbrace{\sign(\bx^{(k)})\sign(\bx^{(k)})}_{=1}\bar\gamma^{(k)}|\\
    &=&  | |\bx^{(k)}|(1-\bar\gamma^{(k)}) -\bar\gamma^{(k)}|\\
    &\leq&  \max\{ |\bx^{(k)}|(1-\bar\gamma^{(k)}) -\bar\gamma^{(k)},
    \bar\gamma^{(k)} - |\bx^{(k)}|(1-\bar\gamma^{(k)}) 
    \}\\
    &\leq&  \max\Bigg\{ \underbrace{\frac{C_2}{k}(1-\frac{C_3}{k}) -\frac{C_3}{k}}_{(*)},
    \frac{C_4}{k} \Bigg\}\\
    \end{eqnarray*}
    and when $k \geq \frac{C_2}{C_3} \iff C_3 \geq \frac{C_2}{k}$,
    \[
    (*) \leq C_2\left(\frac{1}{k} -\frac{1}{k^2}\right) \leq \frac{C_2}{k+1}.
    \]
    Taking also $C_4 \leq \frac{C_2}{4}$,
    \begin{eqnarray*}
    |\bx^{(k+1)}| 
    \leq  \max\left\{ \frac{C_2}{k+1},
    \frac{C_2}{4k} \right\} \leq \frac{C_2}{k+1}\\
    \end{eqnarray*}
    for all $k \geq 1$.
    
    \item Now suppose that there is some $i$ where $\bar\bs_i = \sign(\bx^{(k)}\mb1^T)$. 
    Now since 
\[
\bar \bS = -\sign(\bx^{(k)}\mb 1^T + \bZ A^T)
\]
then this implies that 
$ |\bx^{(k)}| < (\bZ A^T)_i$.
But since 
\[
|(\bZ A^T)_i| \leq \|\bZ\|_\infty \|A\|_\infty q \leq \frac{c}{c+k}(C_1\|A\|_\infty q) \leq \frac{C_2}{4k}, 
\]
this implies that
\[
|\bx^{(k+1)}| \leq \frac{C_2}{4k}(1-\frac{C_3}{k}) + \frac{C_2}{4k} \leq \frac{C_2}{2k}\leq \frac{C_2}{k+1}, \quad \forall k > 1.
\]
\end{itemize}
Thus we have shown the induction step, which completes the proof.
\end{proof}

\begin{lemma}

There exists a finite point $\tilde k$ where for all $k > \tilde k$, 
\[
 \frac{c}{c+k}-\frac{C_4}{k^2} < |\xi_i| < \frac{c}{c+k}+\frac{C_4}{k^2}
\]
for some constant $C_4>0$.
\end{lemma}
\begin{proof}
Our goal is to show that 
\[
\gamma^{(k)} - \frac{C_4}{k^2} \leq \|\bZ\|_\infty \leq  \gamma^{(k)} + \frac{C_4}{k^2}
\]
for some $C_4\geq 0$, and for all $k \geq k'$ for some $k' \geq 0$.
Using the Woodbury matrix identity,
\[
\Gamma(I+A^T\Gamma)^{-1} = \Gamma\left(I - A^T(I+\Gamma A^T)^{-1} \Gamma\right)
\]
and thus
\[
\bZ^{(k)} = \bar \bS\Gamma -\underbrace{\left(\bx^{(k)}\mb 1^T \Gamma + (\bar \bS-\bx^{(k)}\mb 1^T)\Gamma A^T(I+\Gamma A^T)^{-1}\Gamma\right)}_{\mathbf B}.
\]
and thus
\[
   |\bar \bs_i \tilde \gamma_i| - \frac{C_3}{k^2} \leq |\xi_i^{(k)}| \leq |\bar \bs_i \tilde \gamma_i| + \frac{C_3}{k^2}
\]
where via triangle inequalities and norm decompositions,
\[
\frac{C_3}{k^2} = \max_i |\mathbf B_i| \leq \underbrace{|\bx^{(k)}|}_{O(1/k)} \gamma_k + D_4 \gamma_k^2 \|A\|_\infty (I+\Gamma^{(0)}A^T)^{-1} = O(1/k^2).
\]
Finally, since $\bar \bs_i\in \{-1,1\}$, then 
$|\bar \bs_i\tilde \gamma_i| = \tilde\gamma_i$, and in particular,
\[
\frac{c}{c+k+\omega_i}\leq \frac{c}{c+k}
\]
and
\[
\frac{c}{c+k+\omega_i}\geq \frac{c}{c+k+\omega_{\max}} = \frac{c}{c+k} - \frac{c}{c+k}\frac{\omega_{\max}}{c+k+\omega_{\max}} \geq \frac{c}{c+k}-\frac{c\omega_{\max}}{k^2}
\]
Therefore, taking $C_4 = c\omega_{\max}+C_3$ completes the proof.

\end{proof}

\begin{lemma}
There exists some large enough $\tilde k$ where for all $k \geq \tilde k$, it must be that
\begin{equation}
\exists k'\geq k, \quad -\sign(\bx^{(k')})T(\bx^{(k')}) > \frac{\epsilon}{k'}.
\label{eq:lemma-helper-1}
\end{equation}
\end{lemma}

\begin{proof}
Define a partitioning $S_1\cup S_2 = \{1,...,q\}$, where
\[
S_1 = \{i : \xi_i > 0\}, \quad S_2 = \{j : \xi_j\leq 0\}.
\]
Defining $\bar \xi = \frac{c}{c+k}$,
\[
|\sum_{i=1}^q \beta_i \xi_i| = |\sum_{i\in S_1}\beta_i |\xi_i| -\sum_{j\in S_2}\beta_j|\xi_j|| \geq  \left(\bar \xi-\frac{C_4}{k^2}\right)\cdot\left|\sum_{i\in S_1}\beta_i-\sum_{j\in S_2}\beta_j\right|.
\]

By assumption, there does not exist a combination of $\beta_i$ where a specific linear combination could cancel them out; that is, suppose that there exists some constant $\bar \beta$, where for \emph{every} partition of sets $S_1$,$S_2$,
\[
0<\bar\beta :=\min_{S_1,S_2}  |\sum_{i\in S_1}\beta_i-\sum_{j\in S_2}\beta_j|.
\]
Then  
\[
|\sum_{i=1}^q \beta_i \xi_i|  \geq  \left(\frac{c}{c+k}-\frac{C_2}{k^2}\right)\bar\beta \geq \bar\beta\frac{\max\{C_2,c\}}{k}.
\]
Picking $\epsilon = \max\{C_2,c\}$ concludes the proof.
\end{proof}

\end{document}